\documentclass[12pt]{amsart}

\usepackage{amssymb}
\usepackage[margin=3cm]{geometry}

\newtheorem{theorem}{Theorem}
\newtheorem{lemma}[theorem]{Lemma}
\newtheorem{proposition}[theorem]{Proposition}

\theoremstyle{definition}

\newtheorem{example}[theorem]{Example}
\newtheorem{remark}[theorem]{Remark}

\title{Minimal free resolution of generalized repunit algebras}

\author{Isabel Colaço}
\address{Departamento de Matem\'atica e Ci\^encias F\'{\i}sicas, Instituto Polit\'ecnico de Beja, 7800-295 Beja, Portugal}
\email{isabel.colaco@ipbeja.pt}

\author{Ignacio Ojeda}
\address{Departamento de Matem\'aticas, Universidad de Extremadura, 06071 Badajoz, Spain}
\email{ojedamc@unex.es}

\thanks{The authors are partially supported by Proyecto de Excelencia de la Junta de Andaluc\'{\i}a (ProyExcel\_00868) and by Proyecto de investigaci\'on del Plan Propio - UCA 2022-2023 (PR2022-011). The second author is partially supported by grant PID2022-138906NB-C21 funded by MICIU/AEI/10.13039/501100011033 and by ERDF/EU`, and by research group FQM024 funded by Junta de Extremadura (Spain)/FEDER funds.}



\begin{document}

\begin{abstract}
Let $\Bbbk$ be an arbitrary field and let $b > 1, n > 1$ and $a$ be three positive integers. In this paper we explicitly describe a minimal $S-$graded free resolution of the semigroup algebra $\Bbbk[S]$ when $S$ is a generalized repunit numerical semigroup, that is, when $S$ is the submonoid of $\mathbb{N}$ generated by $\{a_1, a_2, \ldots, a_n\}$ where $a_1 = \sum_{j=0}^{n-1} b^j$ and $a_i - a_{i-1} = a\, b^{i-2},\ i = 2, \ldots, n$, with $\gcd(a,a_1) = 1$.
\end{abstract}

\maketitle

\section{Introduction}

Let $\Bbbk[\mathbf{x}] = \Bbbk[x_1, \ldots, x_n]$ be the polynomial ring in $n$ indeterminates over an arbitrary field $\Bbbk$, let $S$ be the numerical semigroup with minimal system of generators $A = \{{a}_1, \ldots, {a}_n\} \subset \mathbb{N}$ (see \cite{libro-ns} for details on numerical semigroups) and let $\Bbbk[S] := \bigoplus_{a \in S} \Bbbk \chi^a$ be the semigroup $\Bbbk-$algebra of $S$.

Considering the ring $\Bbbk[\mathbf x]$ graded by $S$ via $\deg(x_i) = a_i,\ i = 1, \ldots, n,$ we have that the kernel of the $\Bbbk-$algebra homomorphism \[\varphi_A : \Bbbk[\mathbf{x}] \longrightarrow \Bbbk[S], x_i \mapsto \chi^{a_i}\] determines a presentation of $\Bbbk[S]$ as $S-$graded $\Bbbk[\mathbf{x}]-$module. Indeed, the so-called \emph{toric ideal} $I_A := \ker(\varphi_A)$ is known (see, e.g. \cite[Lemma 4.1]{sturmfels}) to be generated by \[\left\{\mathbf{x}^{\mathbf{u}} - \mathbf{x}^{\mathbf{v}}\ \mid\ \sum_{i=1}^n u_i a_i = \sum_{i=1}^n v_i a_i,\ \mathbf{u}=(u_1, \ldots, u_n), \mathbf{v} = (v_1, \ldots, v_n) \in \mathbb{N}^n \right\},\] where $\mathbf x^\mathbf{u} := x_1^{u_1} \cdots x_n^{u_n}$. In particular, it is homogeneous for the grading determined by $S$.

So, if $\left\{f_i := \mathbf{x}^{\mathbf{u}_i} - \mathbf{x}^{\mathbf{v}_i}\ \mid\ i = 1, \ldots, \beta_1 \right\}$ is a minimal generating system of $I_A$ and $\varphi_0$ denotes the corresponding canonical projection, then \[\Bbbk[\mathbf{x}]^{\beta_1} \stackrel{\varphi_1:=(f_1, \ldots, f_{\beta_1})}{\longrightarrow} \Bbbk[\mathbf x] \stackrel{\varphi_0}{\longrightarrow} \Bbbk[\mathbf{x}]/I_A \cong \Bbbk[S] \to 0\] is exact and $S-$graded by suitable degree shiftings of the leftmost free module. Now, one can compute a minimal system of generators of the kernel of $\varphi_1$, say $\{\mathbf{f}_{12}, \ldots, \mathbf{f}_{\beta_2 2}\} \subset \Bbbk[\mathbf x]^{\beta_1}$, so that the sequence  \[\Bbbk[\mathbf{x}]^{\beta_2} \stackrel{\varphi_2 := (\mathbf{f}_{12} \vert \ldots \vert \mathbf{f}_{\beta_2 2})}{\longrightarrow} \Bbbk[\mathbf{x}]^{\beta_1} \stackrel{\varphi_1}{\longrightarrow} \Bbbk[\mathbf x] \stackrel{\varphi_0}{\longrightarrow} \Bbbk[\mathbf{x}]/I_A \cong \Bbbk[S] \to 0\] is exact and, after the appropriate degree shifts, $S-$graded. So, by repeating this process as many times as necessary until reaching $\ker \varphi_p = 0$, which is guaranteed by the Hilbert syzygy theorem (see, e.g., \cite[Theorem 1.13]{Eisenbud}), we obtain a \emph{minimal $S-$graded free resolution of $\Bbbk[S]$}. The minimal free resolution is unique up to isomorphism (see \cite[Section 20.1]{Eisenbud}). The $\beta_i,\ i = 1, \ldots, p$, are called \emph{Betti numbers of $\Bbbk[S]$} (see Remark \ref{Rem 5} for more details).

Computing a minimal free resolution of $\Bbbk[S]$ is possible using Groebner bases techniques. Other related tasks are to characterize the minimal free resolution $S-$graded in terms of the combinatorics within $S$ (see, for example, \cite{collectanea, OjVi1}) or, for special cases of $S$, to describe explicitly a minimal $S-$graded free resolution of $S$ in terms of $S$ basically (see e.g. \cite{philippe}). This article is about the latter.

Let $b$ and $n$ be two integers greater than one and let $S$ be the submonoid of $\mathbb{N}$ generated by $\{a_1, a_2, \ldots \} \subset \mathbb{N}$, where \[a_1 = \sum_{j=0}^{n-1} b^j\ \text{\and}\ a_i - a_{i-1} = a\, b^{i-2},\ i \geq 2,\] for $a \in \mathbb{Z}_+$ relatively prime with $a_1$. In \cite{BCO2}, it is proved that $S$ is a numerical semigroup whose minimal generating system is $A := \{a_1, \ldots, a_n\}$. These numerical semigroups are called \emph{generalized repunit numerical semigroups} (see \cite{BCO,BCO2}) as they generalize the repunit numerical semigroups introduced in \cite{repunit}.

The aim of this paper is to explicitly describe a minimal $S-$graded free of resolution of $\Bbbk[S]$ when $S$ is a generalized repunit numerical semigroup. In what follows, we consider $S$ to be a generalized repunit numerical semigroup and refer $\Bbbk[S]$ as a generalized repunit $\Bbbk-$algebra.

We notice that if $b=1$, then $S$ is generated by an arithmetic sequence. In this case, the $S-$graded free of resolution of $\Bbbk[S]$ is fully described by P. Gimenez et al. in \cite{philippe}. The minimal free resolution of numerical semigroups generated by arithmetic sequences has its own interest as, for instance, the Betti numbers of $\Bbbk[S]$ and the coordinate ring of its tangent cone ring coincide. We emphasize that, by \cite[Corollary 2]{BCO} and \cite[Theorem 3.12]{JaZar}, generalized repunit $\Bbbk-$algebras also have this property.

Finally, we emphasize that in \cite{Matsuoka1, Matsuoka2} similar techniques are applied to families closely related to ours. In particular, in \cite[Section 4]{Matsuoka1} the authors use Eagon-Northcott complexes to compute the Pseudo-Frobenius numbers of numerical semigroups associated to certain determinantal ideals. These ideas are brilliantly generalized in \cite[Section 2.1]{Matsuoka2}. 

\section{The minimal free resolution}\label{Sect2}

Let $b > 1, n > 1$ and $a > 1$ be three fixed integer numbers such that $a$ and $a_1 = \sum_{j=0}^{n-1} b^j$ are relatively prime. With the same notation as in the introduction, let $S$ be the generalized repunit numerical semigroup generated by $A = \{a_1, \ldots, a_n\}$. 

In \cite{BCO} it is proved that $I_A$ is minimally generated by $2 \times 2-$minors of the matrix 
\begin{equation}\label{ecu1}
X := (x_{ij}) = \left(\begin{array}{cccc}
x_1^b & \cdots & x_{n-1}^b & x_n^b \\ x_2 & \cdots & x_n & x_1^{a+1}
\end{array}
\right).
\end{equation}
Therefore, since $I_A$ is a determinantal ideal, the generalized repunit $\Bbbk-$algebra $\Bbbk[S]$ can be resolved by the Eagon-Northcott complex introduced in \cite{EN} and described below.

Let $y_1, y_2$ be two indeterminates and let $M_j$ be the $\Bbbk[\mathbf x]-$submodule of $\Bbbk[\mathbf x][y_1, y_2]$ generated by the monomials in $y_1$ and $y_2$ of degree $j$. Define \[\Bbbk[\mathbf x]^X_{j} := \bigwedge^{j+1} \Bbbk[\mathbf x]^n \otimes_{\Bbbk[\mathbf x]} M_{j-1},\ j = 1, \ldots, n-1,\]
where $\bigwedge^{j+1} \Bbbk[\mathbf x]^n$ is the degree $j+1$ component of the exterior algebra of the free $\Bbbk[x]-$module $\Bbbk[x]^n$. Thus, if $\{\mathbf{e}_1, \ldots, \mathbf{e}_n\}$ is the usual basis of $\Bbbk[\mathbf x]^n$; that is, the basis of $\Bbbk[\mathbf{x}]^n$ such that $\mathbf{e}_i$ has a one in place $i$ and zeros elsewhere, for each $i \in \{ 1, \ldots, n\}$, then the $\Bbbk[\mathbf x]-$module $\bigwedge^{j+1} \Bbbk[\mathbf x]^n$ is generated by $\mathbf{e}_{i_1} \wedge \cdots \wedge \mathbf{e}_{i_{j+1}}$, for each $1 \leq i_1 < \cdots < i_{j+1} \leq n$, for each $j \in \{1, \ldots, n-1\}$. 

Now, since the codimension of $I_A$ is $n-1$, because $I_A$ defines an irreducible monomial curve in the $n-$dimensional affine space over $\Bbbk$, by \cite[Theorem 2]{EN}, we conclude that 
\[0 \to \Bbbk[\mathbf x]^X_{n-1} \stackrel{d_{n-1}}{\longrightarrow} \Bbbk[\mathbf x]^X_{n-2} \stackrel{d_{n-2}}{\longrightarrow} \cdots \stackrel{d_2}{\longrightarrow}\Bbbk[\mathbf x]^X_{1} \stackrel{d_1}{\longrightarrow} \Bbbk[\mathbf x] \longrightarrow \Bbbk[\mathbf x]/I_A \cong \Bbbk[S] \to 0\] is a minimal free resolution of $\Bbbk[S]$, with \begin{equation}\label{ecu2} d_1(\mathbf e_i \wedge \mathbf e_j \otimes 1) = \left| \begin{array}{cc} x_{1i} & x_{1j} \\ x_{2i} & x_{2j} \end{array}\right|,\ \text{for every}\ 1 \leq i < j \leq n,\end{equation}
and 
\begin{equation}\label{ecu3} d_j(\mathbf{e}_{i_1} \wedge \cdots \wedge \mathbf{e}_{i_{j+1}} \otimes y_1^{u_1} y_2^{u_2}) = \sum_{k\stackrel{*}{=}1}^2 \sum_{l=1}^{j+1} (-1)^{l+1} x_{k i_l} \mathbf{e}_{i_1} \wedge \cdots \wedge \widehat{\mathbf{e}_{i_l}}  \wedge \cdots \wedge \mathbf{e}_{i_{j+1}} \otimes y_1^{u_1} y_2^{u_2} y_k^{-1},\end{equation} for every $1 \leq i_1 < \cdots < i_{j+1} \leq n,\ u_1,u_2 \in \mathbb{N}$ such that $u_1+u_2=j-1$ and $j \in \{2, \ldots, n-1\}$, where the asterisk means that we only sum over those $k$ for which $u_k > 0$ and $\widehat{\mathbf{e}_{i_l}}$ means omitting $\mathbf{e}_{i_l}$.  

\begin{lemma}
For each $j \in \{1, \ldots, n-1\}$, the $\Bbbk[\mathbf{x}]-$module $\Bbbk[\mathbf x]^X_j$ is isomorphic to $\Bbbk[\mathbf x]^{j \binom{n}{j+1}}$.
\end{lemma}

\begin{proof}
Since $\bigwedge^{j+1} \Bbbk[\mathbf x]^n$ and $M_{j-1}$ are isomorphic as $\Bbbk[\mathbf{x}]-$modules to $\Bbbk[\mathbf x]^{\binom{n}{j+1}}$ and $\Bbbk[\mathbf{x}]^j$, respectively, for each $j \in \{1, \ldots, n-1\}$, we have that \[\Bbbk[\mathbf x]^X_j = \bigwedge^{j+1} \Bbbk[\mathbf x]^n \otimes_{\Bbbk[\mathbf x]} M_{j-1} \cong \Bbbk[\mathbf x]^{\binom{n}{j+1}} \otimes_{\Bbbk[\mathbf x]} \Bbbk[\mathbf{x}]^j \cong \Bbbk[\mathbf x]^{j \binom{n}{j+1}},\] for each $ j = 1, \ldots, n-1$.
\end{proof}

So, applying the previous lemma, we have the following.

\begin{proposition}\label{prop2}
Let $\phi_0$ be the identity map of $\Bbbk[\mathbf x]$ and, for each $j \in \{ 1, \ldots, n-1\}$, fix a $\Bbbk[\mathbf x]-$module isomorphism $\phi_{j} : \Bbbk[\mathbf x]^X_j \to \Bbbk[\mathbf x]^{j \binom{n}{j+1}}$. If $\beta_j = j \binom{n}{j+1}$ and $\delta_j = \phi_{j-1} \circ d_j \circ \phi_j^{-1},\ j = 1, \ldots, n-1$, then \begin{equation}\label{ecuMFR} 0 \to \Bbbk[\mathbf x]^{\beta_{n-1}} \stackrel{\delta_{n-1}}{\longrightarrow} \Bbbk[\mathbf x]^{\beta_{n-2}} \stackrel{\delta_{n-2}}{\longrightarrow} \cdots \stackrel{\delta_2}{\longrightarrow}\Bbbk[\mathbf x]^{\beta_1} \stackrel{\delta_1}{\longrightarrow} \Bbbk[\mathbf x] \longrightarrow \Bbbk[\mathbf x]/I_A \cong \Bbbk[S] \to 0\end{equation} is a minimal free resolution of $\Bbbk[S]$.
\end{proposition}

For a clearer understanding of Proposition \ref{prop2}, we provide, as an illustrative example, the well-known minimal free resolution of $\Bbbk[S]$ for $n=3$ (see, for example, \cite[Theorem 2.3]{OjPi} or, in broader generality, the Hilbert-Burch theorem \cite[Theorem 20.15]{Eisenbud}).

\begin{example}\label{ejem2}
Let $S$ be the numerical semigroup generated by $a_1 = 1 + b + b^2, a_2 = 1 + b + b^2 + a$ and $a_3 = 1 + b + b^2 + a(1+b)$. In this case, a minimal free resolution of $\Bbbk[S]$ is equal to 
\[
0 \to \Bbbk[\mathbf{x}]^2 \stackrel{\delta_2}{\longrightarrow} \Bbbk[\mathbf{x}]^3 \stackrel{\delta_1}{\longrightarrow} \Bbbk[\mathbf{x}] \longrightarrow \Bbbk[S] \to 0
\]
where $\delta_2$ and $\delta_1$ are the $\Bbbk[\mathbf x]-$module homomorphisms whose matrices with respect to the corresponding usual bases are 
\[
A_2 = \begin{pmatrix} \medskip x_1^{b} & x_2\\ \medskip x_2^b & x_3\\ x_3^b & x_1^{a+1}\end{pmatrix}\quad \text{and}\quad A_1 = \left(\begin{array}{ccc} x_2^bx_1^{a+1}-x_3^{b+1} & -x_1^{a+b+1}+x_2x_3^b & x_1^bx_3-x_2^{b+1} \end{array}\right),
\] respectively.
Clearly, in this case, $\beta_1 = 3$ and $\beta_2 = 2$.
\end{example}

\section{The minimal $S-$graded free resolution}

The minimal free resolution \eqref{ecuMFR}, given in the previous section by using Eagon-Northcott, is not $S-$graded in general. The reason for this is that the maps $\delta_i,\ i = 1, \ldots, n-1$, described in Proposition \ref{prop2}, are not necessarily $S-$homogeneous of degree $0$. 

To achieve a minimal free resolution $S-$graded of $\Bbbk[S]$, we must appropriately shift the free $\Bbbk[\mathbf x]-$modules that appear in $\Bbbk[\mathbf{x}]^{\beta_j},\ j = 1, \ldots, n$, in such a way the maps $\delta_j,\ j = 1, \ldots, n$, defined in Proposition \ref{prop2} become $S$-homogeneous of degree $0$. More precisely, we need to find positive integers $s_{j,k},\ k = 1, \ldots, \beta_j,\ j = 1, \dots, n-1,$ such that the  maps in the minimal free resolution 
\begin{align*}
0 \to & \bigoplus_{k=1}^{\beta_{n-1}} \Bbbk[\mathbf x](-s_{n-1,k}) \stackrel{\delta_{n-1}}{\longrightarrow}
\bigoplus_{k=1}^{\beta_{n-2}} \Bbbk[\mathbf x](-s_{n-2,k}) \stackrel{\delta_{n-2}}{\longrightarrow} \cdots \\ & \cdots \stackrel{\delta_{2}}{\longrightarrow}
\bigoplus_{k=1}^{\beta_{1}} \Bbbk[\mathbf x](-s_{1,k})  \stackrel{\delta_1}{\longrightarrow} \Bbbk[\mathbf x] \longrightarrow \Bbbk[\mathbf x]/I_A \cong \Bbbk[S] \to 0 
\end{align*}
are $S-$homogeneous of degree $0$. Recall that $\Bbbk[\mathbf x](-s)$ means that the basis elements of $\Bbbk[\mathbf x](-s)$ as $\Bbbk[x]-$module, say $1$, have degree $s$. Thus, for example, $x_1 \in \Bbbk[\mathbf x](-s)$ has $S-$degree $a_1+s$.

\begin{example}
By considering the degree-shift isomorphisms
\[\Bbbk[\mathbf x]^2 \cong \Bbbk[\mathbf x](-b\, a_1 - (b+1) a_3) \bigoplus \Bbbk[\mathbf x](-a_2 - (b+1) a_3) \]
and 
\[\Bbbk[\mathbf x]^3 \cong \Bbbk[\mathbf x](- (b+1) a_3) \bigoplus  \Bbbk[\mathbf x](-a_2-b\, a_3) \bigoplus \Bbbk[\mathbf x](-(b+1) a_2)\] in Example \ref{ejem2}, we obtain a minimal $S-$graded free resolution of $\Bbbk[S]$ because these shifts make $\delta_2$ and $\delta_1$ $S-$homogeneous of degree $0$. 
\end{example}

\begin{remark}\label{Rem 5}
Given $j \in \{1, \ldots, n-1\}$, we have that $\operatorname{Tor}^{\Bbbk[\mathbf x]}_j(\Bbbk, \Bbbk[S])_s \neq 0$ if and only if $s = s_{j,k}$ for some $1 \leq k \leq \beta_j$ (see, e.g. \cite[Lemma 1.32]{MS}); in fact, the number of $s_{j,k}$'s that are equal to a given $s \in S$ is $\dim(\operatorname{Tor}^{\Bbbk[\mathbf x]}_j(\Bbbk, \Bbbk[S])_s)$. Summarizing, the integers $s_{j,k}$ are uniquely determined by $\Bbbk[S]$.
\end{remark}

Our goal is to compute the integers $s_{i,k}$. To start, we introduce additional notation. From now on we will write $a_{n+1} = (a+1) a_1$ and $c = b^n - 1 -a$. 

\begin{lemma}\label{lema6}
With the notation above, $b\, a_i = c + a_{i+1},\ i = 1, \ldots, n$.
\end{lemma}

\begin{proof}
Clearly, $b\, a_1 = b \left(\sum_{j=0}^{n-1} b^j\right) = b^n - 1 + a_1 = b^n - 1 - a + a_2 = c + a_2 .$ Now, if $i \in \{2, \ldots, n\}$, then 
\begin{align*}
b\, a_i & = b\, a_1 + b\, a \left(\sum_{j=0}^{i-2} b^j\right) = c + a_2 + a \left(\sum_{j=1}^{i-1} b^j\right)\\ & = c + a_1 + a + a \left(\sum_{j=1}^{i-1} b^j\right) = c + a_1 + a \left(\sum_{j=0}^{i-1} b^j\right) = c + a_{i+1},
\end{align*}
and we are done.
\end{proof}

\begin{proposition}\label{prop7}
The maps in the exact sequence \[\bigoplus_{i=1}^{n-1} \left(\bigoplus_{j=0}^{i-1} \Bbbk[\mathbf{x}](-a_{n-i+1} - b\, a_{n-j})\right) \stackrel{\delta_1}{\longrightarrow} \Bbbk[\mathbf x] \longrightarrow \Bbbk[\mathbf x]/I_A \cong \Bbbk[S] \to 0\] are $S-$homogeneous of degree $0$. In particular, the set $\left\{s_{1,k} \mid k = 1, \ldots, \beta_1 = \binom{n}{2}\right\}$ is equal to \(\left\{c + a_{i_1} + a_{i_2} \mid 1 < i_1 < i_2 \leq n+1 \right\}.\)
\end{proposition}

\begin{proof}
In \cite{BCO} it is proved that $I_A$ is minimally generated by $2 \times 2-$minors of the matrix $X$ defined in \eqref{ecu1}. Then first part follows straightforward from the definition of $d_1$ (see \eqref{ecu2}). Now, the second part is an immediate consequence of Lemma \ref{lema6}.
\end{proof}

Recall that $x \in \mathbb{N} \setminus S$ is said to be a \emph{pseudo-Frobenius element of $S$} if $x + s \in S$ for every $s \in S \setminus \{0\}$. This set is known to be finite and is denoted by $\operatorname{PS}(S)$ (see \cite[Section 2.4]{libro-ns} for more details).

\begin{lemma}\label{lema7}
The set $\{s_{n-1, k} \mid k = 1, \ldots, \beta_{n-1} = n-1\}$ is equal to  
\[
\left\{ k\, c + \sum_{i=2}^{n+1} a_i \mid k \in \{1, \ldots, \beta_{n-1} = n-1\} \right\}.
\]
\end{lemma}

\begin{proof}
By \cite[Corollary 17]{MOT}, we have that \[\operatorname{PF}(S) = \left\{s - \sum_{i=1}^n a_i \mid s \in \{s_{n-1,1}, \ldots, s_{n-1,n-1}\}\right\}.\] Now since, by \cite[Corollary 30]{BCO2}, $\operatorname{PF}(S) = \{ k\, c + a\, a_1\ \mid\ k = 1, \ldots, n-1\}$, our claim follows.
\end{proof}

\begin{proposition}\label{prop8}
The maps in the exact sequence \[0 \to \bigoplus_{k=1}^{n-1} \Bbbk[\mathbf{x}](-s_{n-1,k}) \stackrel{\delta_{n-1}}{\longrightarrow} \bigoplus_{k=2}^{n-1} \left( \bigoplus_{j=1}^{n}\Bbbk[\mathbf{x}](b\, a_j - s_{n-1,k})\right)\] are $S-$homogeneous of degree $0$. Moreover, the set $\left\{s_{n-2,k} \mid k = 1, \ldots, \beta_{n-2} = (n-2)n\right\}$ is equal to \[\left\{k\, c + \sum_{\substack{i=2 \\ i \neq j+1}}^{n+1} a_i \mid k \in \{1, \ldots, n-2\}\ \text{and}\ j \in \{1, \ldots, n\}\right\}.\] \end{proposition}

\begin{proof}
For the first part, it suffices to observe the matrix of the map $d_{n-1}$ defined in \eqref{ecu3} with respect to the usual bases is 
\[
\left(
\begin{array}{ccccccc}
x_1^b & x_2 & 0 & \ldots & 0 & 0 \\
\vdots & \vdots & \vdots & & \vdots & \vdots \\
x_{n-1}^b & x_n & 0 & \ldots & 0 & 0 \\
x_n^b & x_1^{a+1} & 0 & \ldots & 0 & 0 \\
0 & x_1^b & x_2 & \ldots & 0 & 0 \\
\vdots & \vdots & \vdots & & \vdots & \vdots \\
0 & x_{n-1}^b & x_n & \ldots & 0 & 0 \\
0 & x_n^b & x_1^{a+1} & \ldots & 0 & 0 \\ 
\vdots & \vdots & \vdots & & \vdots & \vdots \\ 
0 & 0 & 0 & \ldots & x_1^b & x_2  \\
\vdots & \vdots & \vdots & & \vdots & \vdots \\
0 & 0 & 0 & \ldots & x_{n-1}^b & x_n  \\
0 & 0 & 0 & \ldots & x_n^b & x_1^{a+1}  \\
\end{array}
\right).
\]
For the second part, we first observe that, by Lemma \ref{lema7}, we can suppose $s_{n-1,k} = k\, c + \sum_{i=2}^{n+1} a_i,\ k = 2, \ldots, n-1$. Therefore, by Lemma \ref{lema6}, we have that 
\begin{align*}
s_{n-1, k} - b\, a_j  & = k\, c + \sum_{i=2}^{n+1} a_i - b\, a_j = k\, c + \sum_{i=2}^{n+1} a_i - c - a_{j+1}\\ & = (k-1) c + \sum_{i=2}^{n+1} a_i - a_{j+1} = (k-1) c + \sum_{\substack{i=2 \\ i \neq j+1}}^{n+1} a_i,
\end{align*}
for every $j \in \{1, \ldots, n\}$ and $k \in \{2, \ldots, n-1\}$.
\end{proof}

Now, we can finally state and prove the main theorem about the minimal $S-$graded free resolution of the generalized repunit $\Bbbk-$algebra $\Bbbk[S]$.

\begin{theorem}
The set $B_j := \{s_{j,k} \mid k = 1, \ldots, \beta_j = j \binom{n}{j+1}\}$ is equal to \[B'_j := \left\{ k\, c + a_{i_1} + \cdots + a_{i_{j+1}}\ \mid\ k \in \{1, \ldots, j\}\ \text{and}\ 1 < i_1 < \cdots < i_{j+1} \leq n+1 \right\},\] for every $j \in \{1, \ldots, n-1\}$.
\end{theorem}

\begin{proof}
First, we note that, by Propositions \ref{prop7} and \ref{prop8}, we already know that the result is true for $j = 1$ and $j = n-1$, respectively.

Let $j \in \{2, \ldots, n-2\}$. If we fix a bijection $\sigma_j : B_j \to B'_j$, then we have the isomorphism of free $\Bbbk[\mathbf x]-$modules \[\bigoplus_{k=1}^{\beta_j} \Bbbk[\mathbf{x}](-s_{j,k})  \longrightarrow  F_j := \bigoplus_{1 < i_1 < \cdots < i_{j+1} \leq n+1} \left(\bigoplus_{k = 1}^{j}  \Bbbk[\mathbf x](-k\, c -  a_{i_1} - \cdots - a_{i_{j+1}}) \right) \]
such that the element in the usual basis of left-hand module which has a $1$ at place $\Bbbk[\mathbf{x}](-s_{j,k})$ and zeros elsewhere is sent to the element in the usual basis of $F_j$ which has a $1$ at place $\Bbbk[\mathbf{x}](-\sigma(s_{j,k}))$ and zeros elsewhere . So, if we prove that, for each $j \in \{2, \ldots, n-2\}$, there exist $\Bbbk[\mathbf x]-$module isomorphisms $\phi_j : \Bbbk[\mathbf x]_j^X \to F_j$ and $\phi_{j-1} : \Bbbk[\mathbf x]_{j-1}^X \to F_{j-1}$ such that \[\delta_j := \phi_{j-1} \circ d_j \circ \phi_j^{-1} : F_j \to F_{j-1}\] is $S-$homogeneous of degree $0$, where $d_j : \Bbbk[\mathbf x]_j^X \to  \Bbbk[\mathbf x]_{j-1}^X$ is the $\Bbbk[\mathbf{x}]-$module homomorphism defined in Section \ref{Sect2}, we are done since the $s_{j,k}$ are uniquely defined (see Remark \ref{Rem 5}).

Let $j \in \{2, \ldots, n-2\}$ and let us define the $\Bbbk[\mathbf x]-$module isomorphism $\phi_j : \Bbbk[\mathbf x]_j^X \to F_j$  such that $\phi_j(\mathbf{e}_{i_1} \wedge \cdots \wedge \mathbf{e}_{i_{j+1}} \otimes y_1^{u_1} y_2^{u_2})$ is equal to the element that has a $1$ in the coordinate corresponding to the direct summand $\Bbbk[\mathbf x](-(u_1 + 1)\, c  - a_{i_1+1} - \cdots - a_{i_{j+1}+1})$ and zeros elsewhere (recall that $u_1, u_2 \in \mathbb{N}$ verifies $u_1+u_2 = j-1$). Since, given $j \in \{2, \ldots, n-2\}$ and $l \in \{1, \ldots, j\}$, the $S-$degree of the $k\, i_j-$th entry of the matrix $X$ defined in \eqref{ecu1} is 
\[
\left\{
\begin{array}{lcl}
b\, a_{i_l} & \text{if} & k=1; \\
a_{i_l+1} & \text{if} & k=2,
\end{array}
\right.
\]
we have that the $S-$degree of $\phi_{j-1}(x_{k i_l} \mathbf{e}_{i_1} \wedge \cdots \wedge \widehat{\mathbf{e}_{i_l}}  \wedge \cdots \wedge \mathbf{e}_{i_{j+1}} \otimes y_1^{u_1} y_2^{u_2} y_k^{-1})$ is 
\[
\left\{
\begin{array}{lcl}
b\, a_{i_l} + u_1\, c + a_{i_1+1} + \cdots + a_{i_{j+1}+1} - a_{i_l+1} & \text{if} & k=1; \\
a_{i_l+1} + (u_1+1)\, c + a_{i_1+1} + \cdots + a_{i_{j+1}+1} - a_{i_l+1}  & \text{if} & k=2,
\end{array}
\right.
\]
In the first case, by Lemma \ref{lema6}, we have that 
\[b\, a_{i_l} + u_1\, c + a_{i_1+1} + \cdots + a_{i_{j+1}+1} - a_{i_l+1} = \]
\[c + a_{i_l+1}  + u_1 c + a_{i_1+1} + \cdots + a_{i_{j+1}+1} - a_{i_l+1} = \]
\[(u_1+1) c + a_{i_1+1} + \cdots + a_{i_{j+1}+1}.\]
and, in the second case, we have that \[a_{i_l+1} + (u_1+1) c + a_{i_1+1} + \cdots + a_{i_{j+1}+1} - a_{i_l+1} = \]
\[(u_1+1) c + a_{i_1+1} + \cdots + a_{i_{j+1}+1}.\]
So, the $S-$degree $\phi_{j-1}(x_{k i_l} \mathbf{e}_{i_1} \wedge \cdots \wedge \widehat{\mathbf{e}_{i_l}}  \wedge \cdots \wedge \mathbf{e}_{i_{j+1}} \otimes y_1^{u_1} y_2^{u_2} y_k^{-1})$ is equal to the $S-$degree of $\mathbf{e}_{i_1} \wedge \cdots \wedge \mathbf{e}_{i_{j+1}} \otimes y_1^{u_1} y_2^{u_2}$, for every $j \in \{2, \ldots, n-2\}$ and $l \in \{1, \ldots, j\}.$ Therefore, since
\[ d_j(\mathbf{e}_{i_1} \wedge \cdots \wedge \mathbf{e}_{i_{j+1}} \otimes y_1^{u_1} y_2^{u_2}) = \sum_{k\stackrel{*}{=}1}^2 \sum_{l=1}^{j+1} (-1)^{l+1} x_{k i_l} \mathbf{e}_{i_1} \wedge \cdots \wedge \widehat{\mathbf{e}_{i_l}}  \wedge \cdots \wedge \mathbf{e}_{i_{j+1}} \otimes y_1^{u_1} y_2^{u_2} y_k^{-1},\] for every $j \in \{2, \ldots, n-1\}$,  we conclude that, for our choice of the isomorphisms $\phi_j,\ j = 2, \ldots, n-2$, the maps 
$\delta_j = \phi_{j-1} \circ d_j \circ \phi_j$ are $S-$homogeneous of degree zero, for every $j = 2, \ldots, n-2$. 
\end{proof}

\medskip
\noindent \textbf{COI Statement.} The authors declared that they have no conflict of interest.

\medskip
\noindent \textbf{Acknowledgments.} 
The authors wish to thank Manuel B. Branco for his valuable suggestions that improved the presentation of this article. We also thank Naoyuki Matsuoka for letting us know his interesting results on Pseudo-Frobenius numbers and semigroup ideals associated to certain numerical semigroups.

\end{document}